\documentclass{amsart}
\RequirePackage[utf8]{inputenc}
\usepackage{amsmath,amsthm,amsfonts,amssymb,amscd,amsbsy,multirow,hyperref}
\usepackage[all]{xy}
\usepackage{tikz}
\usepackage{graphicx,pgf,caption,subcaption}
\usepackage{enumerate}
\usepackage{dsfont}
\usepackage{stmaryrd}
\usepackage{ytableau}
\usepackage{amsmath,amsthm,amsfonts,amssymb,amscd,amsbsy,dsfont}
\usepackage{graphicx}
\usepackage{hyperref}

 \newtheorem*{acknowledgment}{Acknowledgment}

\newtheorem{lemma}{Lemma}
\newtheorem{proposition}{Proposition}

\newtheorem{theorem}{Theorem}

\numberwithin{equation}{section}

\allowdisplaybreaks

\title[ Modified Hawking
mass and Rigidity of free boundary minimal disks]{ Rigidity of free boundary minimal disks in\\ mean convex three-manifolds}
		
\author{Rondinelle Batista}
\author{Barnab\'{e} Lima}
\author{Jo\~ao Silva}

\address{Universidade Federal do Piau\'{i} - UFPI, Departamento de Matem\'{a}tica, Campus Petr\^onio Portella, 64049-550, Teresina / PI, Brazil} \email{rmarcolino@ufpi.edu.br}
\email{barnabe@ufpi.edu.br}
\email{vinicius130silva@ufpi.edu.br}

\subjclass[2010]{Primary 53C25, 53C20, 53C21; Secondary 53C65}
\keywords{Free boudary minimal surfaces; · Scalar curvature; Mean curvature;}

\begin{document}

\newcommand{\spacing}[1]{\renewcommand{\baselinestretch}{#1}\large\normalsize}
\spacing{1.2}

\begin{abstract}
The purpose of this article is to study rigidity of free boundary minimal two-disks
that locally maximize the modified Hawking mass on a Riemannian three-manifold with a positive lower bound on its scalar curvature and mean convex boundary.  Assuming the strict stability of $\Sigma$, we prove that a neighborhood of it in $M$ is isometric to one of
the half de Sitter-Schwarzschild space.
\end{abstract}

\maketitle

\section{Introduction and statement of the main results}\label{introd}

In the last decades very much attention has been given  to study Riemannian manifolds with  scalar curvature bounded from below since it has strong connections with theory of minimal surfaces.

Several rigidity results involving the scalar curvature have been obtained assuming the existence of area-minimizing surfaces. In \cite{Cai} Cai and Galloway showed that if a Riemannian three-manifold with nonnegative scalar curvature contains an embedded, two-sided, locally area-minimizing two-torus $\Sigma$, then the metric is flat in some neighborhood of $\Sigma$. Later, in \cite{Brendle} Bray, Brendle and Neves showed that  if $(M^3, g)$ is
a Riemannian three-manifold with scalar curvature $R\geq 2$ and $\Sigma^2\subset M^3$ is an embedded two-sphere which is locally area-minimizing, then $\Sigma$ has area less than
or equal to $4\pi$ and if moreover the equality holds, then with the induced metric $\Sigma$ has constant Gauss curvature equal to $1$ and locally $M$ splits along $\Sigma$. Afterward, Nunes \cite{Nunes} studied the hyperbolic setting; he proved that if $(M^3, g)$ is a Riemannian three-manifold
with scalar curvature $R\geq -2$ and $\Sigma^2\subset M^3$ is a two-sided compact embedded Riemann surface of genus $\gamma(\Sigma)\geq 2$ which is locally area-minimizing, then the area
of  $\Sigma$ with respect to the induced metric is greater than or equal to $ 4\pi(\gamma(\Sigma)-1)$. Moreover, if
equality holds, then $\Sigma$ has constant Gauss curvature equal to $-1$ and locally $M$ splits along $\Sigma$. In the nice  work \cite{Micallef}, Micallef and Moraru unified the approach of the results proved in \cite{Cai}, \cite{Brendle}  and \cite{Nunes}. Finally, similar rigidity results have been obtained for closed stable MOTS  in initial data sets by  Galloway and  Mendes in \cite{mendes}, as well as by Mendes in \cite{abra}.
  
Motivated by these results, Máximo and Nunes \cite{mn} established a local rigidity result for the de Sitter-Schwarzschild space, involving strictly stable minimal surfaces and the Hawking mass. For this, it is important to recall the quasi-local, so-called Hawking mass of a compact surface $\Sigma^2\subset (M^3 ,g)$, which is defined as 
\begin{equation}\label{defimassH}
	m_{H}(\Sigma)=\sqrt{\dfrac{|\Sigma|}{16\pi }}\left(\frac{\chi(\Sigma) }{2}-\frac{1}{16\pi }\int_{\Sigma}\left(H^2+\frac{2}{3}\Lambda\right) d\sigma\right),
\end{equation}
where $\chi(\Sigma)$ is the Euler characteristics of $\Sigma$, $H$ is the mean curvature of $\Sigma$, $|\Sigma|$ is the area of $\Sigma$ and $\Lambda$ is the infimum of the scalar curvature. This  notion  of quasi-local mass plays a crucial role in the proof of the Riemannian Penrose inequality for assymptotically flat manifolds, discovered  independently by Huisken and Ilmanen \cite{ht} and Bray \cite{Bray}.

More precisely, Máximo and Nunes proved the following result.
\begin{theorem}[M\'aximo and Nunes]\label{mn}
Let $(M, g)$ be a Riemannian three-manifold with scalar curvature $R\geq 2$. If $\Sigma\subset M$ is an
embedded strictly stable minimal two-sphere which locally maximizes the Hawking mass, then the Gauss curvature of $\Sigma$ is constant, equal to $1/a^2$
for some $a\in (0, 1)$ and a neighborhood of $\Sigma$ in
$(M, g)$ is isometric to the de Sitter-Schwarzschild space $((-\varepsilon, \varepsilon)\times\Sigma, g_a)$ for some $\varepsilon>0$.
\end{theorem}

In the presence of a nonempty boundary, the  objects of study are the  free boundary minimal surfaces. These surfaces arise as critical points of the area functional for surfaces in a three dimensional manifold $M$ with boundary in $\partial M$. It follows from the first variation area formula that such surfaces intersect $\partial M$ orthogonally.

In this context, Ambrozio \cite{lucas} established the following boundary version of the aforementioned rigidity results.
\begin{theorem}[Ambrozio]
Let $(M, g)$ be a Riemannian three-manifold with boundary $\partial M$. Assume that $R$ and $H^{\partial M}$ are bounded from below.
If $\Sigma$ is a properly immersed, two-sided,
free boundary stable minimal surface, then
$$I (\Sigma)=\frac{1}{2}{\rm inf} R|\Sigma|+{\rm inf} H^{\partial M} \leq 2\pi\chi(\Sigma).$$
Assume that $(M,g)$ has mean convex boundary and  $\Sigma$ is a properly embedded, two-sided, locally area-minimizing free boundary
surface such that $$I (\Sigma) = 2\pi\chi(\Sigma).$$ In addition, if one of the following hypotheses
holds:
\begin{itemize}
\item[(i)] each component of $\partial\Sigma$ is locally length-minimizing in $\partial M$; or
\item[(ii)]${\rm inf}H^{\partial M} = 0$;
\end{itemize}
then there exists a neighborhood of $\Sigma$ in $(M, g)$ that is isometric to $((-\varepsilon,\varepsilon)\times
\Sigma,dt^2+g_{\Sigma})$, where $(\Sigma, g_{\Sigma})$ has constant Gaussian curvature $\frac{1}{2}{\rm inf} R$ and $\partial\Sigma$ has constant geodesic curvature ${\rm inf} H^{\partial M}$ in $\Sigma$.
\end{theorem}
A similar result was obtained by Alaee, Lesourd and Yau in \cite{ala}, in the context of free boundary marginally outer trapped surfaces in initial data sets.

Our main goal is to  present a connection between free boundary minimal surfaces and General Relativity using the modified Hawking mass.  We recall that the modified Hawking mass of a compact  free boundary $\Sigma^2$ in a Riemannian three-manifold $(M^3, g)$ is  defined to be
\begin{equation}\label{defimass}
	\widetilde{m}_{H}(\Sigma)=\sqrt{\dfrac{|\Sigma|}{8\pi }}\left(\chi(\Sigma)-\frac{1}{8\pi }\int_{\Sigma}\left(H^2+\frac{2}{3}\Lambda\right) d\sigma\right),
\end{equation}
where where $\chi(\Sigma)$ is the Euler characteristics of $\Sigma$, $H$ is the mean curvature of $\Sigma$, $|\Sigma|$ is the area of $\Sigma$ and $\Lambda$ is the infimum of the scalar curvature. This notion was introduced by Marquardt in \cite{Mar} to study weak solutions  of the free boundary inverse mean curvature flow, and played an important role in the proof of the Riemannian Penrose inequality for asymptotically flat manifolds with a non-compact boundary given by Koerber in \cite{Koe}.

Before we proceed, let us provide a precise definition of model space of this present paper, namely half de Sitter-Schwarzschild manifold which was first introduced by de Lima; see Remark 4.8 of \cite{dL}. Fix $0<m<\frac{1}{3\sqrt{3}}$, the half de Sitter-Schwarzschild metric with mass parameter $m$ and scalar curvature equal to $2$ is defined as the metric
\begin{equation}
g_{m}=\frac{1}{1-\frac{r^2}{3}-\frac{2m}{r}}dr^2+r^2g_{\mathbb{S}_{+}^2},
\end{equation}
define on $(r_-(m),r_+(m))\times \mathbb{S}_{+}^2$, where $0 < r_{-}(m) < r_{+}(m) < 1 $ are the only positive solution of $f_m(r) = 1 -r^2 -2mr^{-1}$, $f_m(r)>0$ for all $r\in(r_-(m),r_+(m))$
and $g_{\mathbb{S}_{+}^2}$ is the standard metric on $\mathbb{S}_{+}^2$ with constant Gauss curvature equal to $1$. Additionally, this manifold  carries a totally geodesic inner boundary $(r_-(m),r_+(m))\times \mathbb S^1$.

Through a change of variable, the half de Sitter-Schwarzschild   metric can be rewritten as
\begin{equation}\label{metric}
g_{m} = ds^2 + u(s)^2g_{\mathbb{S}_{+}^2},\ \ \  {\rm on}\ \ \  [a, b]\times\mathbb{S}^2,
\end{equation}
where $u: (a, b)\rightarrow (r_-(m),r_+(m))$ is a function that extends continuously to $[a, b]$ with $u(a) = r_-(m)$, $u(b) = r_+(m)$ and $\frac{ds}{dr}=f_m(r)^{-1/2}>0$ on $(r_-(m),r_+(m))$.

After reflection of the metric $g_{m}$, we can define a
complete periodic rotationally symmetric metric  on $\mathbb{R}\times\mathbb{S}_{+}^2$ with 
scalar curvature to $R=2$ and totally geodesic boundary $\mathbb{R}\times\mathbb{S}^1$. Moreover, the function $u$ solves the following second-order nonlinear
differential equation
\begin{equation}\label{edo}
u''(s)=\frac{1}{2}\Big(\frac{1-u'(s)^2}{u(s)}\Big)-\frac{u(s)}{2}.
\end{equation}

Moreover, any positive solutions $u(r)$ to \eqref{edo} that is defined for all  $r\in\mathbb{R}$,  define a periodic rotationally symmetric
metric $g_a = dr^2 + u_a(r)^2g_{\mathbb{S}_+^{2}}$ with constant scalar curvature equal to $2$,
where $a\in (0, 1)$ and $u_a(r)$ satisfies $u_a(0) = a = \min u$ and $u_a'(0) = 0$; we refer to [\cite{nthesis}, Section 2.1] for  details. This
metric is precisely the half de Sitter-Schwarzschild metrics on $\mathbb{R}\times\mathbb{S}_+^{2}$ defined
above.

In our first result we show that slices of the half de Sitter-Schwarzschild are local maxima in the following sense:
\begin{theorem}\label{grafico}
Let $\Sigma_r = \{r\}\times\mathbb{S}_{+}^2$ be a slice of the half de 
Sitter–Schwarzschild manifold $(\mathbb{R}\times\mathbb{S}_{+}^2, g_m)$. 
Then there exists an $\varepsilon=\varepsilon (r)>0$ such that if 
$\Sigma\subset
\mathbb{R}\times\mathbb{S}_{+}^2$ is a free boundary properly embedded 
two-disk, which is a normal graph over $\Sigma_r$ given by $\phi\in C^{2}(\Sigma_r)$ with $||\phi||_{C^{2}(\Sigma_r)}<\varepsilon$, one has
\begin{itemize}
\item[(i)] either $\widetilde{m}_H(\Sigma) < \widetilde{m}_H(\Sigma_r)$;
\item[(ii)] or $\Sigma$ is a slice $\Sigma_s$ for some $s$.
\end{itemize}
\end{theorem}

One should point out that  the problem of investigating the rigidity/flexibility of the  half de Sitter-Schwarzschild manifolds appeared in the good survey due to de Lima [Remark 4.8, \cite{dL}].

In our next result, we establish  the following local rigidity theorem: 
\begin{theorem}\label{teoprincipal}
Let $(M, g)$ be a Riemannian three-manifold with boundary $\partial M$ and satisfies  $R\geq 2$ and $H^{\partial M}\geq 0$.  If $\Sigma$ is a properly
embedded, two-sided, free boundary strictly stable minimal two-disk which locally maximizes the modified Hawking mass, then the Gauss curvature of $\Sigma$ is constant equal to $1/a^2$
for some $a\in (0, 1)$, the geodesic curvature of $\partial\Sigma$ vanishes and a neighborhood of $\Sigma$ in
$(M, g)$ is isometric to the half de Sitter-Schwarzschild space $((-\varepsilon, \varepsilon)\times\Sigma, g_a)$ for
some $\varepsilon>0$.
\end{theorem}

We point out that Thoerem \ref{teoprincipal} is a free boundary version of the results due to Máximo and Nunes \cite{mn}. Furthermore, a relevant observation is that the metric $g_m$ defined in \eqref{metric} converges to the standard product metric $dr^2 + g_{\mathbb{S}_{+}^2}$ on $\mathbb{R}\times\mathbb{S}_{+}^2$ when $a\rightarrow 1$. Moreover, one can easily verify that $\Sigma_0 =\{0\}\times \mathbb{S}_{+}^2$ is a free boundary strictly stable  minimal (in fact, totally geodesic) two-disk of area $2\pi a^2$ in $(\mathbb{R}\times\mathbb{S}_{+}^{2}, g_m)$, for each $a\in (0, 1)$. However, in the standard product metric $dr^2 + g_{\mathbb{S}_{+}^2}$, that is, in the limit as $a\rightarrow 1$, $\Sigma_0$ is only stable and not strictly so. Consequently, our theorem provides, in some sense, a generalization of the theorem due to Ambrozio \cite{lucas}.

The outline of the paper is as follows: In Section \ref{sec2}, we review some key aspects of the stability of free boundary minimal surfaces and derive the variation formulae of the modified Hawking mass. In Section \ref{sec3}, inspired by arguments in \cite{lucas}, \cite{mn} we use the strictly stability and the locally maximizes of the modified Hawking mass to construct a foliation of $M$ in a neighborhood of $\Sigma$ by constant mean curvature (CMC) embedded free boundary  surfaces. Finally, in Section \ref{sec4} we prove Theorems \ref{grafico} and \ref{teoprincipal}.

\section{Preliminaries}\label{sec2}
During this section we will gather some basic facts and key lemmas that will be useful for establishment of the main results. Let $(M,g)$ be a Riemannian three-manifold with nonempty boundary $\partial M$. Denote by $R$ the scalar curvature of $M$ and  denote  the mean curvature of $\partial M$ by $H^{\partial M}$. Let $\Sigma^2$ be a smooth compact manifold with nonempty boundary, and suppose  that $\Sigma$ is properly embedded into $M$, that is, $\Sigma$ is  embedded into $M^3$ and $\partial\Sigma=\Sigma\cap\partial M$.

Moreover, we assume that $\Sigma$ is two-sided, that is, there exists a unit vector field $N$ along $\Sigma$ that is normal to $\Sigma$. Fix a unit normal vector field $X$ outward pointing unit normal for $\partial M$, and let $\nu$ be the outward pointing conormal along $\partial\Sigma$ in $\Sigma$. Let $A(U,V) = g(-\nabla_U N,V)$ be the second fundamental form  of $\Sigma$, and let $\Pi(u,v)=g(\nabla_uX,v)$ be the second fundamental form of $\partial M$ with respect to $-X$. We say that $\Sigma$ is free boundary if $\Sigma$ meets $\partial M$ orthogonally. In other words, $\Sigma$ is free boundary if $\nu = X$ along $\partial\Sigma$.

Let $f:\Sigma\times(-\varepsilon,\varepsilon)\rightarrow M$ be a properly smooth normal variation of $\Sigma$, that is, $f$ is a smooth map such that, 
\begin{itemize}
\item for every $t\in(-\varepsilon,\varepsilon)$, the map $ f_t =f(\cdot,t):\Sigma\rightarrow M$ is a properly embedded into $M$;
\item $f(x,0)=x$ for every $x\in\Sigma$;
\item $\frac{\partial f}{\partial t} (x,0) = \phi(x)N(x)$ for each $x\in\Sigma$, where $\phi\in C^{\infty}(\Sigma)$.
\end{itemize}

We now recall some well-knowns  evolution equations of relevant geometric quantities. For this purpose, we will use the subscript $t$ to denote quantities associated to $\Sigma_t = f_t(\Sigma)$. More precisely, $N_t$ will denote a local unit vector field normal to $\Sigma_t$, $H_t$ the mean curvature of $\Sigma_t$, $\nu_t$ the outward pointing conormal along $\partial\Sigma_t$ and $\phi_t$ is the function on $\Sigma_t$ defined by $\phi_t = g(\frac{\partial f}{\partial t},N_t)$.

First, we recall the variation of the mean curvature
\begin{equation}\label{varimean}
	\left\{\begin{array}{ll}\frac{d}{dt}H(t)\Big{|}_{t=0}=L_{\Sigma}\phi \ \ on\ \ \Sigma,\\\frac{d}{dt}g(N_t, X)\Big{|}_{t=0}=-\frac{\partial\phi}{\partial\nu}+\Pi(N,N)\phi\ \ along\ \  \partial\Sigma,\end{array}\right.
\end{equation}
where $L_{\Sigma}=\Delta_{\Sigma}+Ric(N,N)+|A|^2$ is the Jacobi operator of $\Sigma$, $Ric$ is the Ricci tensor of $M$ and $\Delta_{\Sigma}$ is the Laplace operator of $\Sigma$ with respect to the induced metric from $M$.

In particular, if each $\Sigma_t$ is a constant mean curvature free boundary surface, then
\begin{equation}\label{varimeanCMC}
	\left\{\begin{array}{ll}\frac{d}{dt}H(t)=L_{\Sigma_t}\phi_t \ \ {\rm on}\ \ \Sigma,\\\frac{\partial\phi_t}{\partial\nu_t}=\Pi(N_t,N_t)\phi_t\ \ {\rm along}\ \  \partial\Sigma.\end{array}\right.
\end{equation}

The first variation of area is given by
\begin{equation}\label{1area}
\frac{d}{dt}|\Sigma_t|\Big{|}_{t=0}=-\int_{\Sigma}H\phi d\sigma+\int_{\partial\Sigma}\phi g(N, \nu)ds,
\end{equation}
where  $H = tr A$ is the mean curvature of $\Sigma$ in $M$. From \eqref{1area}, it follows that $\Sigma$ is a critical point for the area functional if and only if $\Sigma$ is minimal with free boundary.  For more details see \cite{lucas}.

If $\Sigma$ is minimal with free boundary, then the second variation of area is given by
\begin{equation}\label{2area}
\frac{d^2}{dt^2}|\Sigma_t|\Big{|}_{t=0}=-\int_{\Sigma}\phi L_{\Sigma}\phi d\sigma +\int_{\partial\Sigma}\Big(\frac{\partial\phi}{\partial\nu}-\Pi(N,N)\phi\Big)\phi ds.
\end{equation}
Alternatively, in terms of the quadratic form this can be written
\begin{equation*}
\frac{d^2}{dt^2}|\Sigma_t|\Big{|}_{t=0}= \mathcal{Q}(\phi,\phi),
\end{equation*}
where the quadratic form $\mathcal{Q}:C^{\infty}(\Sigma)\times C^{\infty}(\Sigma)\rightarrow \mathbb{R}$ is called the index form of $\Sigma$ given by
\begin{equation}\label{index}
\mathcal{Q}(\phi,\psi)=-\int_{\Sigma}\phi(\Delta_{\Sigma}\psi+(Ric(N,N)+|A|^2)\psi) d\sigma +\int_{\partial\Sigma}\phi\Big(\frac{\partial\psi}{\partial\nu}-\Pi(N,N)\psi\Big) ds.
\end{equation}
The boundary Robin condition
$$\frac{\partial\phi}{\partial\nu}=\Pi(N,N)\phi$$
is an elliptic boundary condition for $L_{\Sigma}$, therefore there exists a non-decreasing and diverging
sequence $\lambda_1\leq\lambda_2\leq\cdots\leq\lambda_k\nearrow\infty$ of eigenvalues associated to a $L^2(M, d\sigma)$ orthonormal basis of solutions to the eigenvalue problem
\begin{equation}\label{eingevalue}
\left\{\begin{array}{ll}L_{\Sigma}\phi+\lambda\phi=0\ \ {\rm on}\ \ \Sigma,\\\frac{\partial\phi}{\partial \nu}=\Pi(N,N)\phi\ \ {\rm along}\ \  \partial\Sigma.\end{array}\right.
\end{equation}

Then $\Sigma$ is
called stable if and only if $\mathcal{Q}(\phi,\phi)\geq 0$ for all $\phi\in C^{\infty}(\Sigma)$, where the lowest
eigenvalue of $L_{\Sigma}$ is nonnegative. If the lowest eigenvalue is positive we say the
surface is strictly stable. Therefore, if $\phi$ is an eigenfunction of the Jacobi operator associated
with that first eigenvalue, that is, $\mathcal{Q}(\phi,\phi)=\lambda_1(L_{\Sigma})\int_{\Sigma}\phi^2d\sigma$, then
\begin{equation}\label{strictly}
\lambda_1(L_{\Sigma})\int_{\Sigma}\phi^2d\sigma =-\int_{\Sigma}\phi(\Delta_{\Sigma}\phi+(Ric(N,N)+|A|^2)\phi)d\sigma.
\end{equation} 

Next, we derive first and the second variation formulae for the modified Hawking mass.
\begin{proposition}\label{1variation}
Let $(M, g)$ be a Riemannian three-manifold with boundary $\partial M$ and scalar curvature $R$ bounded below.  Let $\Sigma$ be a properly
embedded, two-sided, free boundary surface. For a given normal variation as in the discussions above, we have
{\setlength\arraycolsep{2pt}
\begin{eqnarray}
		\left. \frac{d}{dt}\widetilde{m}_{H}(\Sigma_{t})\right|_{t=0}&=&-\frac{2|\Sigma|^{-\frac{1}{2}}}{(8\pi )^{\frac{3}{2}}}\int_{\Sigma}H\Delta_{\Sigma}\phi d\sigma+\frac{|\Sigma|^{\frac{1}{2}}}{(8\pi )^{\frac{3}{2}}}\int_{\Sigma}(\Lambda-R)H\phi d\sigma\nonumber \\
		&+&\frac{|\Sigma|^{\frac{1}{2}}}{(8\pi)^{\frac{3}{2}}}\int_{\Sigma}\left[2K_{\Sigma}-\frac{4\pi\chi(\Sigma)}{|\Sigma|}+\frac{1}{2|\Sigma|}\int_{\Sigma}H^2 d\sigma-|A|^2\right]H\phi d\sigma, \nonumber
\end{eqnarray} }
where $\Lambda=\inf R$ and $K_{\Sigma}$ is the Gaussian curvature of $\Sigma$.
\end{proposition}
\begin{proof}
Taking the derivative of the modified Hawking mass given in \eqref{defimass} for the variation $f_t$ of $\Sigma$ provided  previously to obtain
{\setlength\arraycolsep{2pt}
\begin{eqnarray}
		\frac{d}{dt}\widetilde{m}_{H}(\Sigma_{t})&=&\frac{1}{2}\frac{|\Sigma_{t}|^{-\frac{1}{2}}}{(8\pi )^{\frac{1}{2}}}\left(1-\frac{1}{8\pi }\int_{\Sigma_{t}}H_{t}^2d\sigma_{t}-\frac{1}{12\pi }\Lambda|\Sigma_{t}|\right)\frac{d}{dt}d(\sigma_{t}) \nonumber \\
		&+& \frac{|\Sigma_{t}|^{\frac{1}{2}}}{(8\pi)^{\frac{1}{2}}}\left(-\frac{1}{8\pi }\int_{\Sigma_{t}}\left[2H_{t}\frac{d}{d t}(H_{t})d\sigma_{t}+H_{t}^2\frac{d}{dt}(d\sigma_{t})\right]-\frac{1}{12\pi }\Lambda\frac{d}{dt}(d\sigma_{t})\right),\nonumber 
\end{eqnarray}}where $\frac{d}{dt}d\sigma_t$ denotes the first variation of area functional. Hence, using the fact that $\Sigma$ is free boundary at $t=0$ and the Gauss equation
\begin{equation}\label{gauss}
2Ric(N,N)=R-2K_{\Sigma}+H^2-|A|^2,
\end{equation}
the first variation formula for the modified Hawking mass follows by a straightforward computation using the
following identities \eqref{1area} and \eqref{varimean}.
\end{proof}

In order to set the stage for the proof of the second variation formula of the modified Hawking mass, it is crucial to recall the first variation of the Jacobi operator obtained by Máximo and Nunes \cite{mn}.
\begin{proposition}[\cite{mn}, Proposition 6.2] For each function $\psi\in C^{\infty}(\Sigma)$, we have:{\setlength\arraycolsep{2pt}
\begin{eqnarray*}
L_{\Sigma}'(0)\psi&=&\frac{d}{dt}L_{\Sigma_t}\Big{|}_{t=0}\psi  = 2\phi g(A, Hess\psi)+2\psi g(A, Hess\phi)- 2\phi\omega(\nabla\psi)- 2\psi\omega(\nabla\phi)\\
&+&\phi g(\nabla H, \nabla\psi)-Hg(\nabla\psi,\nabla\phi)+2A(\nabla\phi,\nabla\psi) -\psi div_{\Sigma}(div_{\Sigma}\omega)-\phi\psi HK_{\Sigma}\\
&+&\phi\psi HRic(N,N)+\phi\psi H|A|^2+\phi\psi A_{ij}A_{ik}A_{jk}+\phi\psi R_{iNNj}A_{ij}
\end{eqnarray*}}
where $\omega$ is the 1-form on $\Sigma$ defined by $\omega(X) = Ric(X,N)$.
\end{proposition}

Now, following the procedure adopted in \cite{mn}, we shall present the second variation formula, which can be expressed as follows:

\begin{proposition}\label{segundavaria}
	Under the considerations
of Proposition \ref{1variation}. If  $\Sigma \subset M$ is a critical point of the modified Hawking mass, then 
{\setlength\arraycolsep{2pt}
\begin{eqnarray}		
	\left.\frac{d^2}{dt^2}\widetilde{m}_{H}(\Sigma_{t})\right|_{t=0}&=& -\frac{3\widetilde{m}_{H}(\Sigma)}{4|\Sigma|^2}\left(\int_{\Sigma}H\phi d\sigma\right)^2-\frac{2|\Sigma|^{\frac{1}{2}}}{(8\pi)^{\frac{3}{2}}}\int_{\Sigma}((L_{\Sigma}\phi)^2+HL_{\Sigma}'(0)\phi )d\sigma\nonumber \\
	&+&\frac{|\Sigma|^{\frac{1}{2}}}{(8\pi)^{\frac{3}{2}}}\int_{\Sigma}(H^2+\frac{2}{3}\Lambda)(\phi L_{\Sigma}\phi-H^2\phi^2)d\sigma 
		+\frac{4|\Sigma|^{\frac{1}{2}}}{(8\pi)^{\frac{3}{2}}}\int_{\Sigma}H^2 \phi L_{\Sigma}\phi d\sigma \nonumber\\ 
			&-&\frac{\widetilde{m}_{H}(\Sigma)}{2|\Sigma|}\int_{\Sigma}(\phi L_{\Sigma}\phi-H^2 \phi^2+div_{\Sigma}(\nabla_{X}X) )d\sigma\nonumber \\
		&-&\frac{|\Sigma|^{\frac{1}{2}}}{(8\pi)^{\frac{3}{2}}}\int_{\partial \Sigma}(H^2+\frac{2}{3}\Lambda)\left(\frac{\partial \phi }{\partial \nu }-g(N, \nabla_{N}X)\phi\right)\phi ds \nonumber \\
		&+& \frac{\widetilde{m}_{H}(\Sigma)}{2|\Sigma|}\int_{\partial \Sigma}\left(\frac{\partial \phi}{\partial \nu}-g(N, \nabla_{N}X)\phi\right)\phi ds, \nonumber
\end{eqnarray}}
where $X(x)=\frac{\partial f}{\partial t}(x,0)$.
\end{proposition}
\begin{proof}
Once Proposition 1 is established, the above follows after a direct computation
using the first variation formula of the  modified Hawking mass and the fact that the second variation formula of the area element is given by
{\setlength\arraycolsep{2pt}
\begin{eqnarray*}
\frac{d^2}{dt^2}|\Sigma_t|\Big{|}_{t=0}=\int_{\Sigma}(-\phi L_{\Sigma}\phi+H^2\phi-Hdiv_{\Sigma}(\nabla_XX)) d\sigma +\int_{\partial\Sigma}\Big(\frac{\partial\phi}{\partial\nu}-\Pi(N,N)\phi\Big)\phi ds 
\end{eqnarray*}}together with our assumption that $\Sigma$ is a critical point for the modified Hawking mass
{\setlength\arraycolsep{2pt}
\begin{eqnarray*}
\frac{\widetilde{m}_{H}(\Sigma)}{2|\Sigma|}\left(\int_{\Sigma}H\phi d\sigma-\int_{\partial \Sigma}g(\nu, X)ds\right)&=&\frac{-|\Sigma|^{\frac{1}{2}}}{(8\pi)^{\frac{3}{2}}}\Big[\int_{\Sigma}2HL_{\Sigma}\phi d\sigma
-\int_{\Sigma}(H^2+\frac{2\Lambda}{3})H\phi d\sigma \\
&+&\int_{\partial \Sigma}(H^2+\frac{2\Lambda}{3})g(\nu, X)ds\Big].
\end{eqnarray*}}
\end{proof}

\section{Stability result for  modified Hawking mass}\label{sec3}

In this section, we seek to construct a foliation around $\Sigma$ by CMC embedded free boundary surfaces under the assumption of the strictly stability and locally maximizes of the modified Hawking mass, which play an important role in a proof of our main theorem.
\begin{proposition}\label{locally}Let $(M^3, g)$ be a Riemannian manifold with boundary $\partial M$ which satisfies $R\geq 2$ and $H^{\partial M}\geq 0$. If $\Sigma$ is an properly
embedded, two-sided, free boundary strictly stable minimal two-disk which locally maximizes the modified Hawking mass, then
\begin{equation*}
|\Sigma|= \frac{2\pi}{\lambda_1(L_{\Sigma})+1},
\end{equation*}
where $\lambda_1(L_{\Sigma})$ is the first eigenvalue of the stability operator. Moreover, along $\Sigma$, we have  $A = 0$, $R = 2$,  $Ric(N,N) = -\lambda_1(L_{\Sigma})$,  and its Gaussian
curvature $K_{\Sigma} = \frac{2\pi}{|\Sigma|}$, and along $\partial\Sigma$, geodesic curvature is zero  and $H^{\partial M}=0$. In particular, $\Sigma$ is a hemisphere.
\end{proposition}
\begin{proof}
	Since $\Sigma$ is strictly stable, we obtain that  $\lambda_{1}(L_{\Sigma}) \int_{\Sigma}\phi^2 d\sigma\leq \mathcal{Q}(\phi, \phi)$ for any smooth function $\phi $ on $\Sigma$. Choosing $\phi=1$, we
obtain 
	\begin{equation}\label{area<}
		\lambda_{1}(L_{\Sigma})|\Sigma|\leq -\int_{\Sigma}\left(Ric(N, N)+|A|^2\right)d\sigma-\int_{\partial \Sigma}\Pi(N, N)ds.
	\end{equation}
Note that the Gauss equation implies
\begin{equation*}
	Ric(N, N)+|A|^2=\frac{R}{2}-K_{\Sigma}+\frac{|A|^2}{2}\geq 1-K_{\Sigma},
\end{equation*}
where in the last inequality we use that $R\geq 2$.
 
Then, using the Gauss-Bonnet Theorem and the fact that $H^{\partial M} =k_{g}+\Pi(N, N)$ along $\partial\Sigma$,  we deduce
{\setlength\arraycolsep{2pt}
\begin{eqnarray*}
(1+\lambda_{1}(L_{\Sigma}))|\Sigma| &\leq &\int_{\Sigma}K_{\Sigma}d\sigma-\int_{\partial \Sigma}\Pi(N, N)ds\\
&=& 2\pi-\int_{\partial\Sigma}(k_g+\Pi(N,N))ds\\
&=&2\pi-\int_{\partial\Sigma}H^{\partial M}ds,
\end{eqnarray*}}where $k_g$ denotes the geodesic curvature of $\partial\Sigma$ in $\Sigma$. Thus, from $H^{\partial M} \geq 0$ we also have   
\begin{equation}\label{eq}
|\Sigma|\leq\frac{2\pi}{\lambda_{1}(L_{\Sigma})+1}.
\end{equation}

Moreover, if equality holds, all inequalities
above must be equalities. Then, $\mathcal{Q}(1,1)=\lambda_{1}(L_{\Sigma})|\Sigma|$, $\Sigma$ is totally geodesic, $R=2$ and $H^{\partial M}=0$ along $\Sigma$. Next, it follows from $\mathcal{Q}(\phi, \phi)\geq \lambda_{1}(L_{\Sigma})\int_{\Sigma}\phi^2d\sigma$ for any $\phi\in C^{\infty}(\Sigma)$, that 
	\begin{equation*}
		\mathcal{F}(\phi, h):=\mathcal{Q}(\phi, h)-\lambda_{1}(L_{\Sigma})\int_{\Sigma}\phi h d\sigma, 
	\end{equation*}
satisfies $\mathcal{F}(\phi, \phi)\geq 0$ and $\mathcal{F}(1, 1)= 0$ for every $\phi\in C^{\infty}(\Sigma)$. At the same time, one easily verifies that $\mathcal{F}(1,\phi)=0$ for any $\phi\in C^{\infty}(\Sigma)$, and hence we may use this data to deduce
\begin{equation}	
0=\mathcal{F}(1, \phi)	=-\int_{\Sigma}(Ric(N, N)+\lambda_{1}(L_{\Sigma}))\phi-\int_{\partial \Sigma}\Pi(N, N)\phi.
\end{equation}
This allow us to conclude that, along $\Sigma$,
$Ric(N, N) = -\lambda_1(L_{\Sigma})$, its Gaussian curvature $K_{\Sigma} = \frac{2\pi}{|\Sigma|}$ and $\Pi(N, N)=k_g=0$ along $\partial\Sigma$.

Finally, we derive the reverse of inequality \eqref{eq}; using the fact that $\Sigma$ locally maximizes the modified Hawking mass in the Propositon \ref{segundavaria}, we achieve
{\setlength\arraycolsep{2pt}
\begin{eqnarray*}		
  0&\geq & -\frac{2|\Sigma|^{\frac{1}{2}}}{(8\pi)^{\frac{3}{2}}}\int_{\Sigma}(L_{\Sigma}\phi)^2d\sigma 
			-\left(\frac{\widetilde{m}_{H}(\Sigma)}{2|\Sigma|}-\frac{4|\Sigma|^{\frac{1}{2}}}{3(8\pi)^{\frac{3}{2}}}\right)\int_{\Sigma}\phi L_{\Sigma}\phi d\sigma \\		
		&-&\left(\frac{4|\Sigma|^{\frac{1}{2}}}{3(8\pi)^{\frac{3}{2}}}-\frac{\widetilde{m}_{H}(\Sigma)}{2|\Sigma|}\right)\int_{\partial \Sigma}\left(\frac{\partial \phi }{\partial \nu }-g(N, \nabla_{N}X)\phi\right)\phi ds. \\
\end{eqnarray*}}
Furthermore, substituting into the last inequality an eigenfunction of the problem \eqref{eingevalue} satisfying $\int_{\Sigma}\phi^2d\sigma=1$, we have
	\begin{equation*}			
-\frac{2|\Sigma|^{\frac{1}{2}}}{(8\pi)^{\frac{3}{2}}}\lambda_1(L_{\Sigma})^2+\frac{\widetilde{m}_{H}(\Sigma)}{2|\Sigma|}\lambda_1(L_{\Sigma})-\frac{4|\Sigma|^{\frac{1}{2}}}{3(8\pi)^{\frac{3}{2}}}\lambda_1(L_{\Sigma})\leq 0,
	\end{equation*}
using the definition of modified Hawking mass, we conclude that
	\begin{equation*}
		( 8\pi -4|\Sigma|)\lambda_{1}(L_{\Sigma})\leq 4|\Sigma|\lambda_{1}(L_{\Sigma})^2,
	\end{equation*}
and, since $\lambda_{1}(L_{\Sigma})>0$, we infer
	\begin{equation*}
	|\Sigma|\geq \frac{2\pi}{\lambda_{1}(L_{\Sigma})+1}
\end{equation*}
which finishes the proof of the proposition.
\end{proof}

The next result is a crucial step in the proof of the main result. Its proof follows the same arguments used in [Proposition 5.1, \cite{mn}] and [Proposition 10, \cite{lucas}]. One can construct a one-parameter
family $\Sigma(t)$ in the neighborhood of $\Sigma$ such that $\Sigma(t)$ is a free boundary surface and has constant
mean curvature.

\begin{proposition}\label{cmc} Let $(M^3, g)$ be a Riemannian manifold with boundary $\partial M$ which satisfies $R\geq 2$ and $H^{\partial M}\geq 0$. If $\Sigma$ is a properly
embedded, two-sided, free boundary strictly stable minimal two-disk satisfying
\begin{equation*}
|\Sigma|= \frac{2\pi}{\lambda_1(L_{\Sigma})+1},
\end{equation*}
then there exists $\varepsilon > 0$ and a smooth function $\mu : (-\varepsilon,\varepsilon)\times\Sigma \rightarrow\mathbb{R}$ such that
$$\Sigma_t = \{\exp_x(\mu(t, x)N(x)); x \in\Sigma\}$$
is a family of free boundary surfaces with constant mean curvature and $N$ is the unit normal vector field along $\Sigma$. Moreover, the following properties hold:
$$\mu(0, x) = 0,\ \ \ \frac{\partial\mu}{\partial t} (0, x) = 1\ \ \  \text{and}\ \ \ \int_{\Sigma}(\mu(t,x)-t)d\sigma = 0$$
for each $x \in\Sigma$ and for each $t\in (-\varepsilon, \varepsilon)$. In particular, for some smaller $\varepsilon$, $\{\Sigma_t\}_{t\in(-\varepsilon,\varepsilon)}$ is a foliation of a neighborhood of $\Sigma_0=\Sigma$ in $M$.
\end{proposition}
\begin{proof}
Since this proposition is crucial for the establishment of Theorem \ref{teoprincipal}, we include its
proof here for the sake of completeness. 

Considering the notation used in the section \ref{sec2}, let $N$ be the unit normal vector field of $\Sigma$, and let $X$ denote the unit normal vector field of $\partial M$ that coincides with the exterior conormal $\nu$ of $\partial\Sigma$.

For a function $u\in C^{2,\alpha}(\Sigma)$, $0<\alpha<1$, we define $\Sigma_{u}=\{\exp_{x}(u(x)N(x)); x\in\Sigma\}$. Note that $\Sigma_u$ is a properly embedded surface when $||u||_{2,\alpha}$ is sufficiently small.
Next, consider the Banach spaces $Z=\{u \in C^{2,\alpha}(\Sigma); \int_{\Sigma} ud\sigma = 0\}$ and $Y = \{u\in C^{0,\alpha}(\Sigma); \int_{\Sigma} ud\sigma = 0\}$. Given sufficiently small constants $\varepsilon>0$ and $\delta>0$, we define the map $\Upsilon:(-\varepsilon,\varepsilon)\times (B(0,\delta)\cap Z)\rightarrow Y\times C^{1,\alpha}(\partial\Sigma)$  by 
$$\Upsilon(t,u)=\left(H_{\Sigma_{t+u}}-\frac{1}{|\Sigma|}\int_{\Sigma} H_{\Sigma_{u+t}}d\sigma, g(N_{\Sigma_{u+t}},X_{\Sigma_{u+t}})\right),$$
where $B(0,\delta) = \{u\in C^{2,\alpha}(\Sigma); ||u||_{2,\alpha} < \delta\}$, $H_{\Sigma_{t+u}}$ denotes the mean curvature of $\Sigma_{u+t}$, $N_{\Sigma_{u+t}}$ denote the unit normal vector field of $\Sigma_{u+t}$, and $X_{\Sigma_{u+t}}$ denotes the restriction of $X$ to $\partial\Sigma_{u+t}$.

Note that $\Upsilon(0,0)=(0,0)$, because $\Sigma=\Sigma_0$. By Proposition \ref{locally}, the Jacobi operator of $\Sigma$ is given by $L_{\Sigma}=\Delta_{\Sigma}-\lambda_1(L_{\Sigma}).$

Consequently, for each $\upsilon\in Z$, we use \eqref{varimean} to obtain
{\setlength\arraycolsep{2pt}
\begin{eqnarray*}
D\Upsilon (0, 0)\cdot \upsilon &=& \frac{d}{ds}\Big{|}_{s=0}\Upsilon (0, s\upsilon)\\
&=&\frac{d}{ds}\Big{|}_{s=0}\left(H_{\Sigma_{s\upsilon}}-\frac{1}{|\Sigma|}\int_{\Sigma}H_{\Sigma_{s\upsilon}}d\sigma,  g(N_{\Sigma_{s\upsilon}},X_{\Sigma_{s\upsilon}})\right)\\
&=& \left(L_{\Sigma}\upsilon -\frac{1}{|\Sigma|}\int_{\Sigma}\Delta_{\Sigma}\upsilon d\sigma+ \frac{\lambda_1(L_{\Sigma})}{|\Sigma|}\int_{\Sigma}\upsilon d\sigma,-\frac{\partial\upsilon}{\partial\nu}\right)\\
&=& \left(\Delta_{\Sigma}\upsilon-\lambda_1(L_{\Sigma})\upsilon-\frac{1}{|\Sigma|}\int_{\partial\Sigma}\frac{\partial\upsilon}{\partial\nu}ds ,-\frac{\partial\upsilon}{\partial\nu}\right).
\end{eqnarray*}}Now, we claim that $D\Upsilon (0,0)$ is an isomorphism when restricted to $0\times Z$.
To prove this, it suffices to show that there exists a unique function $\varphi\in Z$ solving the following Neumann boundary problem
\begin{equation}\label{nbp}
\left\{\begin{array}{ll}\Delta_{\Sigma}\varphi-\lambda_1(L_{\Sigma})\varphi=f+\frac{1}{|\Sigma|}\int_{\partial\Sigma}zds\ \ {\rm in}\ \ \Sigma,\\\frac{\partial\varphi}{\partial \nu}=z\ \ {\rm on}\ \  \partial\Sigma.\end{array}\right.
\end{equation}
for given $f\in Y$ and $z\in C^{1,\alpha}(\partial\Sigma)$.  So, it suffices to apply [\cite{edp}, Theorem 3.2] or [\cite{Nardi}, Theorem 3.1] to solve the Neumann boundary problem (\ref{nbp}).

Next, we may invoke the Implicit Function Theorem to conclude for  
some smaller $\varepsilon>0$,
there exists $(t,u(t))\in(-\varepsilon,\varepsilon)\times B(0,\delta)$ such that $u(0) = 0$ and
$\Upsilon(u(t), t ) = (0, 0)$ for any $t \in(-\varepsilon,\varepsilon)$. More precisely, the surfaces
$$\Sigma_{t+u(t)}=\{\exp_x((t+u(t)(x))N(x))  ; x\in\Sigma\}$$
are free boundary constant mean curvature surfaces.

Proceeding, it is easy to see that smooth function $\mu:(-\varepsilon,\varepsilon)\times\Sigma\rightarrow\mathbb{R}$ defined by $\mu(t,x)=t+u(t)(x)$ satisfies $\mu(0,x)=0$ for each $x\in\Sigma$ and $\int_{\Sigma}(\mu(t,x)-t)d\sigma=0$ for each $t\in(-\varepsilon,\varepsilon)$. From here it follows that 
\begin{equation}\label{1}
\int_{\Sigma}\frac{\partial\mu}{\partial t}(0,\cdot)d\sigma =|\Sigma|.
\end{equation}
Moreover, since that $\Sigma_{\mu(t,x)}$ is a CMC free boundary surfaces, for every $t\in(-\varepsilon,\varepsilon)$, we have that
$$H_{\mu(t,\cdot)}=\frac{1}{|\Sigma|}\int_{\Sigma}H_{\mu(t,\cdot)}d\sigma.$$ 
Thus, after differentiating at $t = 0$ and invoking \eqref{varimeanCMC}, we deduce
{\setlength\arraycolsep{2pt}
\begin{eqnarray*}
L_{\Sigma}\left(\frac{\partial\mu}{\partial t}(0,\cdot)\right)&=&\frac{1}{|\Sigma|}\int_{\Sigma}(\Delta_{\Sigma}-\lambda_1(L_{\Sigma}))\frac{\partial\mu}{\partial t}(0,\cdot)d\sigma\\
&=&\frac{1}{|\Sigma|}\int_{\Sigma}\frac{\partial}{\partial\nu}\left(\frac{\partial\mu}{\partial t}(0,\cdot)\right)d\sigma-\frac{\lambda_1(L_{\Sigma})}{|\Sigma|}\int_{\Sigma}\frac{\partial\mu}{\partial t}(0,\cdot)d\sigma
\end{eqnarray*}}
and 
$$\frac{\partial}{\partial\nu}\left(\frac{\partial\mu}{\partial t}(0,\cdot)\right)=0\ \ \ {\rm on}\ \ \  \partial\Sigma,$$
since from by Proposition \ref{eq} we get $\Pi(N,N)=0$ on $\Sigma$. From this and by \eqref{1} it follows that
\begin{equation*}\label{cmcvariation}
\left\{\begin{array}{ll}L_{\Sigma}\left(\frac{\partial\mu}{\partial t}(0,\cdot)\right)=L_{\Sigma}(1)\ \ {\rm on}\ \ \Sigma_t,\\
\frac{\partial}{\partial\nu}\left(\frac{\partial\mu}{\partial t}(0,\cdot)\right)=0\ \ {\rm along}\ \  \partial\Sigma_t.\end{array}\right.
\end{equation*}
From here it follows that $\frac{\partial\mu}{\partial t} (0, x) = 1$ for each $x\in\Sigma$ which concludes the proof of the proposition.
\end{proof}

In order to set the stage for the proofs to follow, we use Proposition \ref{cmc} to define a mapping $f_t : \Sigma\rightarrow M$ by $f_t (x) =
\exp_x(\mu(x, t)N(x))$ for each $t \in (-\varepsilon, \varepsilon)$. Let $N_t (x)$ be the unit vector field normal along $\Sigma_t$ such that $N_0(x) = N(x)$ for all $x\in \Sigma$ and let us
denote $d\sigma_t$ and $ds_t$ the element of area of $\Sigma_t$ and $\partial\Sigma$ , respectively, with respect to the induced metric by $f_t$ .

Moreover, let $H(t)$ denote the mean curvature of $\Sigma_t$ with respct to $\nu_{t}$, as well as the lapse function $\rho_{t}:\Sigma_{t}\to \mathbb{R}$ which is defined by  
$$\rho_{t}(x)=\langle N_{t}(x), \frac{\partial}{\partial t}f_{t}(x)\rangle ,$$
for each  $t \in  (-\varepsilon, \varepsilon)$. Consequently, since $\{\Sigma_t\}$ is a foliation of $\Sigma$ by CMC free boundary surfaces, it follows from \eqref{varimean} and \eqref{varimeanCMC} that
\begin{equation}\label{cmcvariation}
\left\{\begin{array}{ll}H'(t)=L_{\Sigma_t}\rho_t\ \ {\rm on}\ \ \Sigma_t,\\\frac{\partial\rho_t}{\partial \nu_t}=g(N_t,\nabla_{
N_t}X)\rho_t\ \ {\rm along}\ \  \partial\Sigma_t,\end{array}\right.
\end{equation}
where $L_{\Sigma_t}$ is the Jacobi operator with associated surface $\Sigma_t$.

From Proposition \ref{cmc}, it is easy to check that
 $\rho_{0}\equiv 1$. Then we use Proposition
\ref{locally} and \eqref{cmcvariation} to conclude that
\begin{equation}\label{sinalH}
	\left. \frac{d}{dt}\right|_{t=0}H(t)=L_{\Sigma}(1)=-\lambda_{1}(L_{\Sigma})<0.
\end{equation}
This implies that we can choose $\varepsilon$ sufficiently small such that $ H(t)>0$ for $t\in(-\varepsilon,0)$ and $H(t)<0$ for $t\in(0,\varepsilon)$, hence $H$ is a decreasing function on $t$. 
In order to prove the main Theorem, we need to provide the monotonicity  of the modified Hawking mass along this foliation. This is the content of the following lemma.

\begin{lemma}\label{stablecmc} Let $\{\Sigma_{t}\}_t$ be a family of surfaces obtained in the Proposition \ref{cmc}. Then
{\setlength\arraycolsep{2pt}
\begin{eqnarray}
		\int_{ \Sigma_{t}}(Ric (N_{t}, N_{t})&+&|A_{t}|^2)\rho_{t}d\sigma_{t}=\overline{\rho_{t}}\int_{\Sigma_{t}}(Ric(N_{t}, N_{t})+|A_{t}|^2)d\sigma_{t}+H'(t)\theta(t, x)\nonumber \\
		&+&\overline{\rho_{t}}\int_{\Sigma_{t}}\frac{|\nabla \rho_{t}|^2}{\rho_{t}^2}d\sigma_{t}
		+\overline{\rho_{t}}\int_{\partial \Sigma_{t}}g(N_{t}, \nabla_{N_{t}}X)ds_{t}
		-\int_{\Sigma_{t}}(\Delta_{\Sigma_{t}}\rho_{t})d\sigma_{t},\nonumber
\end{eqnarray}}
where $\overline{\rho_{t}}=\frac{1}{|\Sigma_{t}|}\int_{\Sigma_{t}}\rho_{t}d\sigma_{t}$ and $\theta(x, t)$ is a non-positive function.
\end{lemma}
\begin{proof}
	 Since  $\rho_{0}\equiv 1$, by continuity we can assume $\rho_t>0$ in a neighborhood. Thus, multiplying equation \eqref{cmcvariation} by $1/\rho_t$ and integrating over $\Sigma_t$ we conclude that
{\setlength\arraycolsep{2pt}  
	\begin{eqnarray*}
	H'(t)\int_{\Sigma_{t}}\frac{1}{\rho_{t}}d\sigma_{t}&=&\int_{\Sigma_{t}}\frac{\Delta_{\Sigma_{t}}\rho_{t}}{\rho_{t}}d\sigma_{t}+\int_{\Sigma_{t}}(Ric(N_{t}, N_{t})+|A_{t}|^2)d\sigma_{t}\\
	&=&\int_{\partial \Sigma_{t}}\frac{1}{\rho_{t}}\frac{\partial \rho_{t}}{\partial \nu_{t}}ds_{t}+\int_{\Sigma_{t}}\frac{|\nabla \rho_{t}|^2}{\rho_{t}^2}d\sigma_{t}+\int_{\Sigma_{t}}(Ric(N_{t}, N_{t})+|A_{t}|^2)d\sigma_{t}.
\end{eqnarray*}}
Multiplying the last equation by $\overline{\rho_{t}}=\frac{1}{|\Sigma_{t}|}\int_{\Sigma_{t}}\rho_{t}d\sigma_{t}$ and subtracting it from the integral of equation \eqref{cmcvariation}, we find:
{\setlength\arraycolsep{2pt}  
\begin{eqnarray*}
	H'(t)\left(\overline{\rho_{t}}\int_{\Sigma_{t}}\frac{1}{\rho_{t}}d\sigma_{t}-|\Sigma_{t}|\right)&=&\overline{\rho_{t}}\int_{\partial \Sigma_{t}}\frac{1}{\rho_{t}}\frac{\partial \rho_{t}}{\partial \eta_{t}}ds_{t}+\overline{\rho_{t}}\int_{\Sigma_{t}}\frac{|\nabla \rho_{t}|^2}{\rho_{t}^2}d\sigma_{t}\\
	&+&\overline{\rho_{t}}\int_{\Sigma_{t}}(Ric(N_{t}, N_{t})+|A_{t}|^2)d\sigma_{t}
	-\int_{\Sigma_{t}}\Delta_{\Sigma_{t}}\rho_{t}d\sigma_{t}\\
	&-&\int_{ \Sigma_{t}}(Ric (N_{t}, N_{t})+|A_{t}|^2)\rho_{t}d\sigma_{t}.\nonumber
\end{eqnarray*}}Hence, it follows from $\frac{\partial \rho_{t}}{\partial \nu_{t}}=g(N_{t}, \nabla_{N_{t}}X)\rho_{t}$ that:
{\setlength\arraycolsep{2pt}  
\begin{eqnarray}
		\int_{ \Sigma_{t}}(Ric (N_{t}, N_{t})&+&|A_{t}|^2)\rho_{t}d\sigma_{t}=\overline{\rho_{t}}\int_{\Sigma_{t}}(Ric(N_{t}, N_{t})+|A_{t}|^2)d\sigma_{t}+H'(t)\theta(t, x)\nonumber \\
		&+&\overline{\rho_{t}}\int_{\Sigma_{t}}\frac{|\nabla \rho_{t}|^2}{\rho_{t}^2}d\sigma_{t}
		+\overline{\rho_{t}}\int_{\partial \Sigma_{t}}g(N_{t}, \nabla_{N_{t}}X)ds_{t}
		-\int_{\Sigma_{t}}(\Delta_{\Sigma_{t}}\rho_{t})d\sigma_{t},\nonumber
\end{eqnarray}}where $\theta(t, x)=|\Sigma_{t}|-\overline{\rho_{t}}\int_{\Sigma_{t}}\frac{1}{\rho_{t}}d\sigma_{t}$ is a non-positive function. This concludes the proof of the lemma.
\end{proof}

\section{Proof of the Results}\label{sec4}

\subsection{Proof of Theorem \ref{grafico}}
\begin{proof}
To begin, let $(M=\mathbb{R}\times\mathbb{S}^2_{+},g_m)$ be the half de Sitter-Schwarzschild with mass parameter $m>0$. Now, consider the double  $(\widetilde{M},\widetilde{g}_m)$ of $(M,g_m)$ along $\partial M$. More precisely, 
$\widetilde{M} = M\times\{0, 1\}/\sim $, where $(x, 0)\sim (x, 1)$ for all $x\in\partial M$, and $\widetilde{g}_m(x, j) = g_m(x)$ for all $x\in M$ and $j = 0, 1$.  It follows that $\partial M=\mathbb{R}\times\mathbb{S}^1$ is totally geodesic that $(\widetilde{M},\widetilde{g}_m)$ is a $C^{\infty}$ Riemannian manifold and $(\widetilde{M}=\mathbb{R}\times\mathbb{S}^2,\widetilde{g}_m)$ is one de Sitter-Schwarzschild.  

Next, suppose $\Sigma$ is a normal graph over slice $\Sigma_r=\{r\}\times \mathbb{S}_+^2$ given
by $\phi\in C^2(\Sigma_r)$ with $||\phi||_{C^2(\Sigma_r)} < \varepsilon$,  which meets $\partial M$ orthogonally. It easy to check that the modified Hawking mass satisfies $\widetilde{m}_{H}(\Sigma_r)=m$ for any slice of $M$ and $\widetilde{\Sigma}\subset\widetilde{M}$ the double of $\Sigma$, is a closed surface given as a graph over the slice $\widetilde{\Sigma}_r=\{r\}\times\mathbb{S}^2$ that satisfies  $|\widetilde{\Sigma}|_{\widetilde{g}_m} = 2|\Sigma|_{g_m}$ and $|\widetilde{\Sigma}_r|_{\widetilde{g}_m} = 2|\Sigma_r|_{g_m}$. This implies that $m_{H}(\widetilde{\Sigma})=2\widetilde{m}_{H}(\Sigma)$ and $m_{H}(\widetilde{\Sigma}_r)=2\widetilde{m}_{H}(\Sigma_r)$ where $m_H(\cdot)$ and $\widetilde{m}_H(\cdot)$ stand, respectively, for the Hawking mass and the modified Hawking mass. So, it suffices to apply Theorem 1.2 in \cite{mn} to conclude that one has either $m_H(\widetilde{\Sigma}) < m_H(\widetilde{\Sigma}_r)$ or $\widetilde{\Sigma}$ is a slice $\widetilde{\Sigma}_s$ for some $s$. It follows that 
\begin{itemize}
\item[(i)] either $\widetilde{m}_H(\Sigma) < \widetilde{m}_H(\Sigma_r)$;
\item[(ii)] or $\Sigma$ is a slice $\Sigma_s$ for some $s$,
\end{itemize}
which finishes the proof of the theorem.
\end{proof}

\subsection{Proof of Theorem \ref{teoprincipal}}
\begin{proof}	
Let $(M^3, g)$  be a Riemannian manifold and $\Sigma\subset M$ be a free boundary surface under the hypotheses of Theorem \ref{teoprincipal}. As a consequence of Propositions \ref{locally} and \ref{cmc}, the Jacobi opertator of $\Sigma$ is given by  $L_{\Sigma}=\Delta-\lambda_{1}(\Sigma)$, and there is a family of free boundary surfaces with constant mean curvature $\{\Sigma_t\}_{|t|<\varepsilon}$ around $\Sigma=\Sigma_{0}$. Thus, we may invoke Proposition \ref{1variation} and the first variation of the modified Hawking mass, to infer 
{\setlength\arraycolsep{2pt}
\begin{eqnarray*}
	\frac{d}{dt}\widetilde{m}_{H}(\Sigma_{t})
	&=& -\frac{1}{2}\frac{|\Sigma_{t}|^{-\frac{1}{2}}}{(8\pi)^{\frac{1}{2}}}\left(1-\frac{1}{8\pi}\int_{ \Sigma_{t}}(H_{t}^2+\frac{4}{3})d\sigma_{t}\right)\left(\int_{ \Sigma_{t}}H_{t}\rho_{t}d\sigma_{t}-\int_{\partial \Sigma_{t}}g(\nu_{t}, X)ds_{t}\right)\nonumber \\
	&+&\frac{|\Sigma_{t}|^{\frac{1}{2}}}{(8\pi)^\frac{1}{2}}\left(-\frac{1}{8\pi}\int_{ \Sigma_{t}}\left[2H_{t}(\Delta_{\Sigma_{t}}\rho_{t}+Ric(N_{t}, N_{t})\rho_{t}+|A_{t}|^2\rho_{t})\right]d\sigma_{t}\right.\nonumber \\
	&+&\left. \frac{1}{8\pi}\int_{ \Sigma_{t}}(H_{t}^2+\frac{4}{3})H_{t}\rho_{t}d\sigma_{t}-\frac{1}{8\pi}\int_{\partial \Sigma_{t}}(H_{t}^2+\frac{4}{3})g(\nu_{t}, X)d s_{t} \right)\\
	&=&\frac{|\Sigma_{t}|^{\frac{1}{2}}}{(8\pi)^{\frac{3}{2}}}H_t\Big[-\frac{4\pi}{|\Sigma_t|}\int_{\Sigma_{t}}\rho_{t}d\sigma_{t}+\left(\frac{3H_{t}^2}{2}+2\right)\int_{\Sigma_{t}}\rho_{t}d\sigma_{t}
	-2\int_{ \Sigma_{t}}\Delta_{\Sigma_{t}}\rho_{t}d\sigma_{t}\\
	&-&2\int_{ \Sigma_{t}}(Ric(N_{t}, N_{t})+|A_{t}|^2)\rho_{t}d\sigma_{t}\Big],
\end{eqnarray*}}where, in the last equality, we use that $\{\Sigma_t\}_t$ is a family of CMC free boundary surfaces. 

In conjunction with Lemma \ref{stablecmc} and \eqref{gauss}, one sees that
{\setlength\arraycolsep{2pt} 
 \begin{eqnarray*}
\frac{d}{dt}\widetilde{m}_{H}(\Sigma_{t})&=&\frac{-|\Sigma_{t}|^{\frac{1}{2}}}{(8\pi)^{\frac{3}{2}}}H_t\Big[4\pi\overline{\rho_{t}}-\left(\frac{3H_{t}^2}{2}+2\right)|\Sigma_t|\overline{\rho_{t}}
	+2\overline{\rho_{t}}\int_{ \Sigma_{t}}(Ric(N_{t}, N_{t})+|A_{t}|^2)d\sigma_{t}\\
	&+&2H'(t)\theta(t,x)+\overline{\rho_{t}}\int_{\Sigma_{t}}\frac{|\nabla\rho_{t}|^2}{\rho_{t}^2}d\sigma_{t}+2\overline{\rho_{t}}\int_{\partial \Sigma_{t}}g(N_{t}, \nabla_{N_{t}}X)ds_t\Big]\\
	&=&\frac{-|\Sigma_{t}|^{\frac{1}{2}}}{(8\pi)^{\frac{3}{2}}}H_t\Big[4\pi\overline{\rho_{t}}-2\overline{\rho_{t}}\int_{\Sigma_t}K_{\Sigma_t}d\sigma_{t}
	+\overline{\rho_{t}}\int_{ \Sigma_{t}}(R-2)d\sigma_{t}+2H'(t)\theta(t,x)\\
	&+&\overline{\rho_{t}}\int_{ \Sigma_{t}}\left(|A_t|^2-\frac{H_t^2}{2}\right) d\sigma_{t}+\overline{\rho_{t}}\int_{\Sigma_{t}}\frac{|\nabla\rho_{t}|^2}{\rho_{t}^2}d\sigma_{t}+2\overline{\rho_{t}}\int_{\partial \Sigma_{t}}g(N_{t}, \nabla_{N_{t}}X)ds_t\Big],
\end{eqnarray*}}and hence, by utilizing the identity $H^{\partial M} = k_g + \Pi(N_t,N_t)$ and the Gauss-Bonnet Theorem, we obtain
{\setlength\arraycolsep{2pt}
\begin{eqnarray}\nonumber
\frac{d}{dt}\widetilde{m}_{H}(\Sigma_{t})&=&-\frac{|\Sigma_{t}|^{\frac{1}{2}}}{(8\pi)^{\frac{3}{2}}}H_t\Big[2\overline{\rho_t}\int_{\partial\Sigma_t}H^{\partial M}ds_t+\overline{\rho_{t}}\int_{ \Sigma_{t}}(R-2)d\sigma_{t}+2H'(t)\theta(t,x)\\\label{da}
&+&\overline{\rho_{t}}\int_{ \Sigma_{t}}\left(|A_t|^2-\frac{H_t^2}{2}\right) d\sigma_{t}+\overline{\rho_{t}}\int_{\Sigma_{t}}\frac{|\nabla\rho_{t}|^2}{\rho_{t}^2}d\sigma_{t}\Big].
\end{eqnarray}}

Since $\rho_{0}(x)=1$ for all $\in\Sigma$, we can choose $\varepsilon>0$ sufficiently small such that $\rho_{t}(x)>0$ for each $x\in\Sigma_t$ and $t\in (-\varepsilon,\varepsilon)$. Thus, it suffices to use $H^{\partial M}\geq 0$  jointly with  $R\geq 2$ to conclude that $	\frac{d}{dt}\widetilde{m}_{H}(\Sigma_{t})\geq 0$ for $t \in \left[0, \varepsilon\right)$ and $	\frac{d}{dt}\widetilde{m}_{H}(\Sigma_{t})\leq 0$ for  $t \in \left(-\varepsilon, 0\right]$. This implies that
\begin{equation*}
	\widetilde{m}_{H}(\Sigma_{t})\geq \widetilde{m}_{H}(\Sigma),
\end{equation*}
for all $t \in (-\varepsilon, \varepsilon)$. Consequently, taking into account that $\Sigma$ locally maximizes the modified Hawking mass, we  concluded that  $	\frac{d}{dt}\widetilde{m}_{H}(\Sigma_{t})\equiv 0$. This immediately guarantees that $\Sigma_{t}$ is umbilic, $R=2$ along $\Sigma_t$, and $H^{\partial M}=0$ on $\partial \Sigma_{t}$.  Then, we use \eqref{da} to infer
\begin{equation*}
	\frac{\overline{\rho_{t}}}{2}\int_{\Sigma_{t}}\frac{|\nabla\rho_{t}|^2}{\rho_{t}^2}d\sigma_{t}+H'(t)\theta(t, x)=0
\end{equation*} 
for each $t \in (-\varepsilon, \varepsilon)$. On the other hand, from \eqref{sinalH} and Lemma \ref{stablecmc} we have $H'(t)\theta(t, x)\geq 0$,  this enables  to arrive at $\rho_t=1$  $\forall t\in(-\varepsilon, \varepsilon)$. Therefore, up to isometry, we can choose a small neighbourhood of $\Sigma$ such that the
metric is given by $g=dt^2+g_{\Sigma_t}$, where $g_{\Sigma_t}$ is the induced metric by the isometry $f(x, t)=\exp_{x}(tN(x))$.

Now, it suffices to apply Theorem 3.2 in \cite{Huisken} to conclude that 
{\setlength\arraycolsep{2pt}  
 \begin{eqnarray*}
	\frac{\partial}{\partial t}g_{\Sigma_{t}}&=&-2\rho_{t}A_{t}\\
	&=&-H_{t}g_{\Sigma_{t}},
\end{eqnarray*}}
where we have used $\rho_{t}\equiv 1$, $H_{t}$ is constant and $\Sigma_{t}$  is umbilic 
for all $ t \in (-\varepsilon, \varepsilon)$.

Consequently, $g_{\Sigma_{t}}=u_{a}(t)^2g_{\mathbb{S}_{+}^2}$ for all $ t \in (-\varepsilon, \varepsilon)$, where $u_{a}(t)=ae^{-\frac{1}{2}\int_{0}^{t}H(s)ds}$ and 
$a^2=\frac{|\Sigma|}{2\pi}$.

Therefore, the induced metric by isometry $f(x, t)=\exp_{x}(tN(x))$ implies that $g=dt^2+u_{a}(t)^2g_{\mathbb{S}_{+}^2}$ on $\Sigma \times \left(-\varepsilon, \varepsilon\right)$. Moreover, it is easy to see that the function $u_{a}(t)$ is a solution of \eqref{edo}. By uniqueness of the solution to the ODE, we conclude that $g$ is exactly the half de Sitter-Schwarzschild metric  
on $\Sigma \times \left(-\varepsilon, \varepsilon\right)$. Thus, the proof is
completed.
\end{proof}

\begin{acknowledgment}
The authors would like to thank the referee for the valuable suggestions that improved the paper. They would also like to extend special thanks to Cicero T. Cruz for his very helpful comments during the preparation of this article. The first and second authors were partially supported by CNPq/Brazil [Grant: 422900/2021-4], while the third author was partially supported by PAPG/FAPEPI/\linebreak Brazil [Grant: 030/2021].
\end{acknowledgment}

\end{document}